\documentclass[11pt]{amsart}
\usepackage{amsmath}
\usepackage{amssymb}
\usepackage{amsfonts}
\usepackage{latexsym}
\usepackage{color}
\usepackage{graphicx}

\newcommand{\be}{\begin{equation}}
\newcommand{\en}{\end{equation}}
\newcommand{\bea}{\begin{eqnarray}}
\newcommand{\ena}{\end{eqnarray}}
\newcommand{\beano}{\begin{eqnarray*}}
\newcommand{\enano}{\end{eqnarray*}}
\newcommand{\bee}{\begin{enumerate}}
\newcommand{\ene}{\end{enumerate}}

\newcommand{\ad}{^{\mbox{\scriptsize $\dag$}}}

\newcommand{\mc}{\mathcal}
\newcommand{\mb}{\mathbb}
\newcommand{\A}{\mathfrak{A}}

\newcommand{\h}{\mathcal{H}}

\newcommand{\gF}{{\mathfrak F}}
\newcommand{\cF}{{\mathcal F}}
\newcommand{\B}{{\mc B}}
\def\L{{\mathcal L}}
\newcommand{\Lc}{{\mc L}}

\newcommand{\D}{{\mc D}}

\newcommand{\E}{{\mc E}}

\newcommand{\M}{{\mathfrak M}}

\newcommand{\up}{\upharpoonright}

\newtheorem{defn}{Definition}[section]
\newtheorem{thm}[defn]{Theorem}
\newtheorem{prop}[defn]{Proposition}
\newtheorem{lemma}[defn]{Lemma}

\newtheorem{example}[defn]{Example}

\newtheorem{rem}[defn]{Remark}

\def\x{\relax\ifmmode {\mbox{*}}\else*\fi}
\newcommand{\beex}{\begin{example}$\!\!${\bf }$\;$\rm }
\newcommand{\enex}{ \end{example}}
\newcommand{\berem}{\begin{rem}$\!\!${\bf }$\;$\rm }
\newcommand{\enrem}{ \end{rem}}
\newcommand{\bedefi}{\begin{defn}$\!\!${\bf }$\;$\rm }
\newcommand{\findefi}{\end{defn}}
\newcommand{\pa}{partial \mbox{*-algebra}}

\newcommand{\ha}{^{\ast}}
\newcommand{\CN}{{\mathbb C}}
\newcommand{\ip}[2]{\left\langle {#1}\left|{#2}\right.\right\rangle}

\newcommand{\gl}{{\mathfrak L}}
\newcommand{\LDD}{\gl(\D,\D^\times)}

\newcommand{\LD}{\Lc^\dagger(\D)}

\newcommand{\LBDD}[1]{\mbox{${\gl}_{\textsf{B}}^{#1}(\D,\D^\times)$}}

\def\H{{\mathcal H}}

\newcommand{\restr}[1]{_{\!\up{#1}}}
\def\LG{{\mathfrak L}}

\newcommand{\LSS}{{\LG}(\S,\S^\times)}
\def\S{\mathcal S}
\begin{document}
\title[Operators in RHS: spectral properties]
{Operators in Rigged Hilbert spaces: some spectral properties }

\author{Giorgia Bellomonte}
\author{Salvatore di Bella}
\author{Camillo Trapani}
\address{Dipartimento di Matematica e Informatica,
Universit\`a di Palermo, I-90123 Palermo, Italy}
\email{bellomonte@math.unipa.it}
\email{salvatore.dibella@math.unipa.it}
\email{camillo.trapani@unipa.it}

\date{\today}

\begin{abstract} A notion of resolvent set for an operator acting in a rigged Hilbert space $\D \subset \H\subset \D^\times$ is proposed. This set depends on a family of intermediate locally convex spaces living between $\D$ and $\D^\times$, called interspaces. Some properties of the resolvent set and of the corresponding multivalued resolvent function are derived and some examples are discussed.

\end{abstract}

\maketitle

\section{Introduction}\label{sect_introd}

Spaces of linear maps acting on a {\em rigged Hilbert space} (RHS, for short)
$$
\D \subset \H \subset \D^\times
$$
have often been considered in the literature both from a pure mathematical point of view \cite{lass1, lass2, schmu, ctrev} and for their applications to quantum theories (generalized eigenvalues, resonances of Schr\"odinger operators, quantum fields...) \cite{ baum, costin, civitagadella, delamadrgadella, madrid, etfields, parravicini}.
The spaces of test functions and the distributions over them constitute relevant examples of rigged Hilbert spaces and operators acting on them  are a fundamental tool in several problems in Analysis (differential operators with singular coefficients, Fourier transforms) and also provide  the basic background for the study of the problem of the multiplication of
distributions by the duality method \cite{ober, agnese, tratschi2}.

Before going forth, we fix some notations and basic definitions.

Let $\D$ be a dense linear subspace of Hilbert space $\H$ and $t$ a locally convex topology on
$\D$, finer than the topology induced by the Hilbert norm. Then the space $\D^\times$  of all continuous conjugate linear functionals on
$\D[t]$, i.e., the conjugate dual of $\D[t]$, is a linear vector space and {\em contains} $\H$, in the sense that $\H$ can be
identified with a subspace of
$\D^\times$. These
identifications imply that the sesquilinear form $B( \cdot , \cdot )$ that puts $\D$ and $\D^\times$ in duality is an extension of the inner product of $\D$; i.e. $B(\xi, \eta) = \ip{\xi}{\eta}$, for every $\xi, \eta \in \D$ (to simplify notations we adopt the symbol $\ip{\cdot}{\cdot}$ for both of them). The space $\D^\times$ will always be
considered as endowed with the {\em strong dual topology} $t^\times= \beta(\D^\times,\D)$. The
Hilbert space $\H$ is dense in
$\D^\times[t^\times]$.

We get in this way a {\em Gelfand triplet} or {\em rigged Hilbert space} (RHS)
\begin{equation}\label{eq_rhs}
\D[t] \hookrightarrow  \H \hookrightarrow\D^\times[t^\times],
\end{equation}
where $\hookrightarrow $ denotes a continuous embedding with dense range. As it is usual, we will systematically read \eqref{eq_rhs} as a chain of inclusions and we will write $\D[t] \subset  \H \subset\D^\times[t^\times]$ or $(\D[t] , \H ,\D^\times[t^\times])$ for denoting a RHS.

Let $\LDD$ denote the vector space of all continuous linear maps from $\D[t]$ into  $\D^\times[t^\times]$. In $\LDD$ an involution $X \mapsto X\ad$ can be introduced by the equality
$$ \ip{X\xi}{ \eta}= \overline{\ip{X\ad \eta}{ \xi}}, \quad \forall \xi, \eta \in \D.$$  Hence $\LDD$ is a *-invariant vector space. As we shall see, $\LDD$ can be made into a partial *-algebra by selecting an appropriate family of intermediate spaces ({\em interspaces}) between $\D$ and $\D^\times$
and, for this reason, this paper is a continuation of the study on the spectral properties of locally convex quasi *-algebras or partial *-algebras on which several results of  a certain interest have been recently obtained, see e.g. \cite{ABT,ATTbo,ATTsp, betra, betraquasi, ctba}.

\medskip The problem we want to face in this paper is that of giving a reasonable notion of spectrum of an operator $X\in \LDD$; where {\em reasonable} means that it gives sufficient information on the behavior of the operator. Indeed, we propose a definition of resolvent set which is closely linked to the intermediate structure of {interspaces} that can be found between $\D$ and $\D^\times$, as it happens in many concrete examples: a spectral analysis can be performed each time we fix one of these families. Actually,
the definition of resolvent set we will give depends  on the choice of a family $\gF_0$ of interspaces. This is not a major problem if we take into account the problem that originated the spectral analysis of operators in Hilbert spaces. If, for instance, $A \in {\mc B}(\H)$ (the C*-algebra of bounded operators in Hilbert space $\H$), looking for the resolvent set of $A$ simply means looking for the $\lambda$'s in ${\mb C}$ for which the equation
$$ A\xi-\lambda \xi = \eta$$ has a unique solution $\xi\in \H$, for every choice of $\eta \in \H$, with $\xi$ depending continuously on $\eta$.

The same problem can be posed in the framework of rigged Hilbert spaces. For instance, between $\S(\mb R)$, the Schwartz space of rapidly decreasing $C^\infty$ functions, and $\S^\times(\mb R)$, the space of tempered distributions, live many classical families of spaces like Sobolev spaces, Besov spaces, Bessel potential spaces, etc. Let us call  $\gF_0$ one of these families. Then, finding solutions of the equation
$$ X\xi - \lambda \xi = \eta ,$$ with $X \in {\LG}(\S(\mb R),\S^\times(\mb R))$ should be intended in a more general sense: there exists a continuous extension of $X$ to a space $\E\in \gF_0$  where solutions of our equation do exist. The fact that $\xi \in \E$ means that the solution satisfies regularity conditions milder than those needed for $\xi$ to belong to $\S(\mb R)$.

Rigged Hilbert spaces are a relevant example of {\em partial inner product} ({\sc Pip-}) spaces \cite{jpact_book}. A {\sc Pip}-space $V$ is characterized by the fact that the inner product is defined only for {\em compatible pairs} of elements of the space $V$. It contains a complete lattice of subspaces (the so called {\em assaying subspaces}) fully determined by the compatibility relation. An assaying subspace is nothing but an interspace, in the terminology adopted here. The point of view here is however different: we start from a RHS and look for convenient families of interspaces for which certain properties are satisfied. Nevertheless, we believe that an analysis similar to that undertaken here
 could also be performed in the more general framework of {\sc Pip}-spaces,
but this problem will not be considered here.

The paper is organized as follows. In Section \ref{sect_preliminaries} we collect some basic facts on rigged Hilbert spaces and operators on them. In Section \ref{sect_resolvent} we introduce the resolvent and spectrum of an operator $X\in \LDD$. This definition, as announced before, depends on the choice of a family $\gF_0$ of interspaces living between $\D$ and $\D^\times$ and the crucial assumption is that the operator $X$ extends continuously to some of them.
Section \ref{sect_Hilbert} is devoted to elements of $\LDD$ that can be considered also as closable operators in Hilbert space.
In particular, we give an extension of Gelfand theorem on the existence of generalized eigenvectors of a symmetric operator $X\in \LDD$ having a self-adjoint extension in the Hilbert space $\H$. We will prove, under the assumptions that $\D= \D^\infty(A)$, where $A$ is a self-adjoint operator in $\H$ having a Hilbert-Schmidt inverse,  that the operator  $X$ has a complete set of generalized eigenvectors without requiring, as done in Gelfand theorem, that $X$ leaves $\D$ invariant. Moreover, it is shown that these generalized eigenvectors all belong to a certain element of the chain of Hilbert spaces generated by $A$. Finally, in Section \ref{sect_examples} we collect some examples.

\section{Notations and preliminaries}\label{sect_preliminaries}

For general aspects of the theory of \pa s and of their representations, we refer to the monograph
\cite{ait_book}.
For reader's convenience, however, we   repeat here the essential definitions.

A \pa\ $\A$ is a complex vector space with conjugate linear
involution  $\ha $ and a distributive partial multiplication
$\cdot$, defined on a subset $\Gamma \subset \A \times \A$,
satisfying the property that $(x,y)\in \Gamma$ if, and only if,
$(y\ha ,x\ha )\in   \Gamma$ and $(x\cdot y)\ha = y\ha \cdot x\ha $.
From now on, we will write simply $xy$ instead of $x\cdot y$ whenever
$(x,y)\in \Gamma$. For every $y \in \A$, the set of left (resp.
right) multipliers of $y$ is denoted by $L(y)$ (resp. $R(y)$), i.e.,
$L(y)=\{x\in \A:\, (x,y)\in \Gamma\}$, (resp. $R(y)=\{x\in \A:\, (y,x)\in \Gamma\}$). We denote by $L\A$ (resp.
$R\A$)  the space of universal left (resp. right) multipliers of
$\A$.
In general, a \pa\ is not associative.

The {\em unit} of  partial *-algebra $\A$, if any, is an element $e\in \A$ such that $e=e\ha$, $e\in R\A\cap L\A$ and $xe=ex=x$, for every $x\in \A$.

\medskip Let $\D[t]\subset \H \subset \D^\times[t^\times]$ be a RHS
and $\LDD$ the vector space of all continuous linear maps from $\D[t]$ into  $\D^\times[t^\times]$.

  To every $X \in \LDD$ there corresponds a separately continuous sesquilinear form $\theta_X$ on $\D \times \D$ defined by
\begin{equation}\label{theta}
\theta_X (\xi,\eta)= \ip{X\xi}{\eta}, \quad \xi ,\, \eta \in \D.
\end{equation}
 The space of all {\em jointly} continuous sesquilinear forms on $\D \times \D$ will be
 denoted with ${\sf B}(\D,\D)$.
We denote by $\LBDD{}$ the subspace of  all $X \in \LDD$ such that $\theta_X \in {\sf B}(\D,\D)$.

We denote by $\gl^\dagger(\D)$ the *-algebra consisting of all $X\in \LDD$ such that $X\D \subseteq \D$ and $X^\dag \D \subseteq \D$.

Let $\D[t] \subset \H \subset \D^\times[t^\times]$ be a rigged Hilbert space and $\E[t_\E]$ a locally convex space
such that
\begin{equation}\label{eq_interspace}\D[t] \hookrightarrow \E[t_\E] \hookrightarrow \D^\times[t^\times]. \end{equation} Let $\E^\times$ be the conjugate dual of $\E[t_\E]$  endowed with its own strong dual topology $t^\times_\E$.
Then by
duality, $\E^\times$ is continuously embedded in $\D^\times$ and the embedding has dense range. Also $\D$ is
continuously embedded in $\E$,  but in this case the image of $\D$ is not necessarily dense in $\E^\times$ \cite[Example 10.2.21]{ait_book},
unless $\E$ is endowed with the Mackey topology $\tau(\E,\E^\times)=: \tau_{\E}$; in which case we say that $\E$ is an {\em interspace}. If $\E, \cF$ are interspaces and $\E\subset \cF$, then $\tau_\cF$ is coarser than $\tau_\E$.

Let $\E$, $\cF$ be interspaces. Let us define
$$\mc{C}(\E, \cF):=\{ X\in \LDD:\, \exists Y\in  \gl(\E, \cF), \, Y\xi=X\xi, \forall \xi \in \D\},$$
where $\gl(\E, \cF)$ denotes the vector space of all continuous linear maps from $\E[\tau_\E]$ into $\cF[\tau_\cF]$.
It is clear that $X \in \mc{C}(\E, \cF)$ if and only if it has a continuous extension $X_\E:\E[\tau_\E] \to \cF[\tau_\cF]$. In particular, if $X\in \mc{C}(\E, \cF)$, then $X \in \mc{C}(\E, \D^\times)$. The continuous extension of $X$ from $\E$ into $\D^\times$ clearly coincides with $X_\E$. Obviously, if $X,Y\in \mc{C}(\E, \D^\times)$, then $(X+Y)_\E= X_\E+Y_\E$.

If $X\in \mc{C}(\E,\cF)$ then $X_\E \in \gl(\E, \cF)$, hence there exists a Mackey continuous linear map $X_\E^\ddag: \cF^\times \to \E^\times$ such that
$$ \ip{X_\E\xi}{\eta}= \overline{\ip{X_\E^\ddag \eta}{\xi}}, \quad \forall \xi \in \E,  \forall \eta \in \cF^\times.$$ Since the sesquilinear form which puts all pairs $(\E, \E^\times)$ in duality extends the inner product of $\D$, it follows that $X_\E^\ddag$ is an extension of $X\ad$ to $\cF^\times$ with values in $\E^\times$. Hence $X\ad \in \mc{C}(\cF^\times, \E^\times)$. Interchanging the roles of $X$ and $X\ad$ (and recalling that every interspace carries its Mackey topology) we obtain
\begin{equation}\label{eq_adjoint}
X\in \mc{C}(\E,\cF) \Leftrightarrow X\ad \in \mc{C}(\cF^\times, \E^\times).
\end{equation}

Moreover, a simple comparison of topologies shows that
if $\E, \cF$ are interspaces and $Y\in \gl(\E, \cF)$, then there exists $X \in \LDD$ such that $X=Y\upharpoonright_\D$.

\medskip
Let now $X,Y \in \LDD$ and assume  there exists an interspace $\E$ such
that $Y \in \gl(\D,\E)$ and $X \in \mc{C}(\E,\D^\times)$; it would then be natural to define
$$
X\cdot Y \xi = X_\E (Y\xi), \quad \xi \in \D.
$$
 However, this product is not   well-defined, because it may depend on the choice of the interspace $\E$.

 \bedefi
A family $\mathfrak F$ of interspaces in the rigged Hilbert space\
\linebreak
$\left(\D[t], \H, \D^\times[t^\times]\right)$  is called a {\em multiplication  framework} if

\begin{itemize}
\item[(i)]$\D \in \mathfrak F$;
\item[(ii)]
$\forall \,\E\in \mathfrak F$, its conjugate dual $\E^\times$ also belongs to $\mathfrak F$;
\item[(iii)]
$\forall\, \E,\cF \in \mathfrak F$, $\E\cap \cF \in \mathfrak F$.
 \end{itemize}
 \findefi

\bedefi
\label{def-10226}
Let ${\mathfrak F}$ be a multiplication framework in the rigged Hilbert space
$\left(\D[t], \H, \D^\times[t^\times]\right)$. The product $X \cdot Y$ of two elements of $\LDD$ is defined, with
respect to
${\mathfrak F}$, if there exist three interspaces $\E,\cF, {\mathcal G}\in {\mathfrak F}$ such that
$X\in  {\mathcal C}(\cF,{\mathcal G})$ and $Y\in  {\mathcal C}(\E,\cF)$. In this case, the multiplication $X
\cdot Y$ is defined by
$$
 X \cdot Y = \left( X_{\cF} Y_\E \right)\up \D
$$
or, equivalently, by
$$
 X \cdot Y \xi = X_{\cF}Y\xi, \quad \xi \in \D.
$$

\findefi

Actually, the product so defined does not depend on the particular choice of the interspaces $\E,\cF, {\mathcal G}\in
{\mathfrak F}$ but it may depend (and it does!) on the choice of ${\mathfrak F}$. For a more detailed analysis see \cite{ait_book,kurst1}.

As shown in \cite[Theorem 10.2.30]{ait_book}, we have
\begin{thm}
\label{theo-10230}
Let  ${\mathfrak F}$ be a  multiplication framework  in the  rigged Hilbert space\ $\left(\D[t], \H, \D^\times[t^\times]\right)$. Then
$\LDD$, with the multiplication defined above, is a (non-associative) partial *-algebra.
\end{thm}

\section{Inverses and resolvents} \label{sect_resolvent}
Let us consider an element $X\in \LDD$. We denote by $\mc{R}(X)$ the range of $X$.
If  $X$ is injective and $\mc{R}(X)=\D^\times$, then there exists a linear map $X^{-1}: \D^\times \to \D$ such that $XX^{-1}=I_{\D^\times}$ and $X^{-1}X=I_{\D}$.
If $X^{-1}$ is continuous, then its restriction to $\D$ is a bounded operator and it leaves $\D$ invariant.
Conditions for the continuity of $X^{-1}$ are given in \cite[Sect. 38]{koethe}.

The next proposition shows that, even though it is natural, the notion of invertibility considered above is too restrictive.
\begin{prop}\label{prop_nogo} Let $\D[t] \hookrightarrow \H \hookrightarrow \D^\times[t^\times]$ be a rigged Hilbert space and $X\in \LDD$ a linear bijection.
Then there exists a triplet of Hilbert spaces $\H_X \hookrightarrow \H \hookrightarrow \H_X^\times$ such that $\D \subseteq \H_X$ and $\D^\times \subseteq \H_X^\times$.
\end{prop}
\begin{proof} Since $X$ is a bijection, the linear operator $X^{-1}: \D^\times \to \D$ is well-defined. Then we can introduce an inner product on $\D^\times$ by
\begin{equation}\label{eq_ip} \ip{\zeta}{\zeta'}_X= \ip{X^{-1}\zeta}{X^{-1}\zeta'}, \quad \zeta, \zeta'\in \D^\times.\end{equation}
By \eqref{eq_ip} it follows easily that $X$ is an isometry of $\D[\|\cdot\|]$ onto $\D^\times [\|\cdot\|_X]$ so it extends to a unitary operator of $\H$ onto the Hilbert space completion $\H_X^\times$ of $\D^\times [\|\cdot\|_X]$. Since the embedding of $\H$ into $\H_X^\times$ is continuous, the conjugate dual $(\H_X^\times)^\times=:\H_X$ of $\H_X^\times$, with respect to the inner product of $\H$ contains $\D$ as dense subspace.
\end{proof}

\medskip Since the existence of global inverses of operators of $\LDD$ is a so strong condition, one may try to exploit the intermediate structure between $\D$ and $\D^\times$ for a more appropriate definition of the inversion procedure.

The fact that, once fixed a
multiplication framework $\gF$, $\LDD$ becomes a partial *-algebra [Theorem \ref{theo-10230}] suggests, say, an {\em algebraic} definition: $Y \in \LDD$ is the inverse of $X\in \LDD$ if \begin{equation}\label{cond_five}\mbox{$X\cdot Y$ and $Y\cdot X$ are well-defined and $X\cdot Y\xi=Y\cdot X\xi=\xi$, $\forall\xi\in\D$}. \end{equation} This equality, however, does not define $Y$ uniquely, because of possible lack of associativity.

\smallskip Let $X\in \LDD$ and $\gF_0$ a family of interspaces.
Assume that there exist $\E, \cF\in \gF_0$ such that $X \in \mc{C}(\E,\cF)$. If the extension $X_\E$ is bijective from $\E$ into $\cF$, then $X_\E^{-1}$ exists. If $X_\E^{-1}$ is continuous from $\cF$ onto $\E$, then its restriction to $\D$ is automatically continuous from $\D[t]$ into $\D^\times[t^\times]$; i.e. $X_\E^{-1}\upharpoonright_\D \in \LDD$ and, moreover, $X_\E^{-1}\upharpoonright_\D\in \mc{C}(\cF,\E)$.
If this is the case, and if $\gF_0$ is a multiplication framework, then \eqref{cond_five} holds. So that $X_\E^{-1}\upharpoonright_\D$ is the  algebraic inverse of $X$. The converse may fail to be true. For this reason there is no need, in what follows, to consider $\gF_0$ as a multiplication framework.

\berem \label{rem_inverses}By \eqref{eq_adjoint} we know that $X\in \mc{C}(\E,\cF)$ if, and only if, $X\ad \in \mc{C}(\cF^\times, \E^\times)$; we denote by $X\ad_{\cF^\times}$ the continuous extension of $X\ad$ to $\cF^\times$. Assume in particular that $\E, \cF$ are Hilbert spaces and $X \in  \mc{C}(\E,\cF)$. Then, if $X_\E$ has a continuous inverse,  $X^\dag_{\cF^\times}$ also has  a continuous inverse and $((X_\E)^{-1})^\dag=(X\ad_{\cF^\times})^{-1}$.
\enrem

\berem Assume that there exists a second pair $\E', \cF'$, such that $X\in \mc{C}(\E', \cF')$ and $X_{\E'}$ has a continuous inverse $X_{\E'}^{-1}$ from $\cF'$ onto $\E'$. If $\gF_0$ is a multiplication framework, then $\E\cap \E'\in \gF_0$ and $\D$ is dense in $\E\cap \E'$ with the projective topology \cite[Proposition 10.2.24]{ait_book}. But $X_\E^{-1}$ and $X_{\E'}^{-1}$ need not have the same restriction to $\D$, unless $\L:=X_\E(\E\cap \E')$ is an interspace of $\gF_0$ and $X \in \mc{C}(\E\cap \E', \L)$.
\enrem

\bedefi \label{defn_geneigenvalues1}Let $X \in \LDD$ and $\lambda \in {\mb C}$.
We say that $\lambda$ is a {\em generalized eigenvalue} of $X$ if there exists an interspace $\E$
such that $X$ has a continuous extension $X_\E$ from $\E[t_\E]$ into $\D^\times[t^\times]$ and $X_\E-\lambda I_\E$ is not injective.  Any nonzero vector $\xi\in N(X_\E-\lambda I_\E)\subset \E$ is called a {\em generalized eigenvector}. If $\E=\D$ we say that $\lambda$ is an eigenvalue of $X$ and elements of $N(X_\D-\lambda I_\D)$ are called eigenvectors.
\findefi

\berem \label{rem_globinverse}If $X-\lambda I_\D$ has a {\em global} inverse $(X-\lambda I_\D)^{-1}:\D^\times \to \D$ and $X\in \mc{C}(\E, \cF)$, for some $\E, \cF\in \gF$, with $\D\subsetneqq \E$, then there exists $\zeta \in \E\setminus\{0\}$ such that $(X_\E-\lambda I_\E)\zeta =0$. Indeed, let $\xi'\in \E\setminus \D$, then,  $(X_\E -\lambda I_\E) \xi'=(X-\lambda I_\D)\xi$ for a unique $\xi\in \D$. Hence $(X_\E -\lambda I_\E)(\xi'-\xi)=0$, with $\xi'-\xi\neq 0$. This implies that $\lambda$ is (also) a generalized eigenvalue.
\enrem

\beex\label{op p} Let $\S:=\S(\mb R)$ be the Schwartz space of rapidly decreasing $C^\infty$ functions and $\S^\times:=\S^\times(\mb R)$ its conjugate dual (tempered distributions) and consider the very familiar example of the operator $p=-i d/dx \in \LSS$. It is clear that $p$ has no eigenvalues in $\S$. But, since $p \in \gl^\dagger(\S)$, has a continuous extension $\widetilde{p}$ to the whole of $\S^\times$. For every $\lambda\in {\mb R}$, the function $f_\lambda (x)= e^{i\lambda x}$ is an  eigenvector of $\widetilde{p}$ corresponding to the eigenvalue $\lambda$. Thus, every real number $\lambda$ is a generalized eigenvalue of $p$.

This is just an example of a more general situation: every symmetric operator $A$ on a nuclear rigged Hilbert space $\D \subset \H\subset \D^\times$, having a self-adjoint extension to $\H$ and invariant in $\D$ possesses a \emph{complete} set of generalized eigenvectors \cite[Ch.IV, Sect.5, Theorem 1]{gelfand}. We will come back to this point in Section \ref{sect_Hilbert}.
\enex

\bedefi Let $X \in \LDD$ and $\gF_0$ be a family of interspaces.
The {\em $\gF_0$-resolvent set} of $X$, $\varrho^{\gF_0}(X)$, consists of the set of complex numbers $\lambda$ satisfying the following conditions:

there exist $\E, \cF\in \gF_0$, with $\E\subseteq \cF$, such that
\begin{itemize}
\item[(i.1)] $X\in \mc{C}(\E, \cF)$ and $(X_\E-\lambda I_\E )\E=\cF$;
\item[(i.2)] $(X_\E-\lambda I_\E )^{-1}$ exists and it is continuous from $\cF[\tau_\cF]$ onto $\E[\tau_\E]$.
\end{itemize}
The set $\sigma^{\gF_0}(X):= {\mb C}\setminus \varrho^{\gF_0}(X)$ will be called the \em{$\gF_0$-spectrum} of $X$.

\findefi

\berem The assumption of continuity in condition (i.2) can be omitted if we suppose that $\E, \cF$ are Banach spaces; in this case, in fact, the inverse mapping theorem guarantees the continuity of $(X_\E-\lambda I_\E )^{-1}$. \enrem

{\berem If the topology of $\D$ is equivalent to the initial Hilbert norm, then $\D^\times=\H$ and, in this case, $X\in \LDD$ if and only if $\overline{X}\in {\B(\H)}$. The unique possible interspace is $\H$ itself. Hence the spectrum of $X$ coincides with the usual spectrum defined in ${\B(\H)}$.
\enrem}

A crucial point in the previous definition is that there could exist different couples of interspaces for which the requirements of our definition are true. This point requires a careful analysis, which will be performed later.

\bigskip
In order to study the $\gF_0$-resolvent of an operator $X\in \LDD$ it is useful to introduce the notion of {\em regular point} of $X$, in analogy with what is usually done for operators in Hilbert spaces.
We will do this by supposing that $\gF_0$ is a family of interspaces whose elements are Hilbert spaces. In this case, we will prefer the notation ${\mc B}(\E, \cF)$ to $\gl(\E, \cF)$ as a remainder of the fact that its elements are bounded operators from $\E$ into $\cF$. The norm of the Banach space ${\mc B}(\E, \cF)$ will be denoted by $\|\cdot\|_{\E,\cF}$.

\medskip
Let now $\E, \cF \in \gF_0$ and $X \in {\mc B}(\E, \cF)$. The inverse of $X$, when it exists, is, clearly, the unique $Y\in {\mc B}(\cF, \E)$ such that
$$ XY=I_{\cF}, \qquad YX=I_\E.$$
 The set of all invertible elements of ${\mc B}(\E, \cF)$ will be denoted by $G({\mc B}(\E, \cF))$.

\bedefi
A complex number $\lambda$ is called a $\gF_0$-regular point for $X\in\LDD$ if there exist $\E$, $\cF\in \mathfrak{F}_0$, with $\E \subseteq \cF$, and $c_\lambda$, $d_\lambda\in {\mb R}^+$ such that:
\begin{equation} \label{eq_one} c_\lambda \|\xi\| _\E \leq \|(X-\lambda I)\xi\|_\cF \leq d_\lambda \|\xi\| _\E , \qquad \forall \xi\in \D.\end{equation}

The set of regular points of $X$ will be denoted by $\pi(X)$. \findefi
Clearly, $\lambda \in \pi(X)$ if and only if there exist $\E$, $\cF\in \mathfrak{F}_0$, such that $(X-\lambda I)\in \mathcal{C}(\E, \cF)$ and both $(X-\lambda I)$ and its extension $(X_\E-\lambda I_\E)$ to $\E$ are injective.

If $X \in \mathcal{C}(\E, \cF)$, we denote by $\pi^{\E, \cF}(X)$ the set of $\lambda \in {\mb C}$ for which \eqref{eq_one} holds, with $\E, \cF$ fixed.
It is clear that  $\lambda \in \pi^{\E, \cF}(X)$ if, and only if, the extension $X_\E-\lambda I_\E$ of $X-\lambda I$ to $\E$ satisfies
\begin{equation} \label{eq_oneone} c_\lambda \|\xi\| _\E \leq \|(X_\E-\lambda I_\E)\xi\|_\cF \leq d_\lambda \|\xi\| _\E,  \qquad \forall \xi\in \E.\end{equation}

\bedefi
Let $X\in\LDD$ and $\lambda\in \pi^{\E, \cF}(X)$. We call $d_\lambda ^{\,\E, \cF}(X) := {\rm dim\,}\mc{R}(X_\E-\lambda I_\E)^\perp$ (the orthogonal being taken in $\cF$) the {\em defect number} of $X$ at $\lambda$ w. r. to $\E,\cF$.
\findefi
Let $X \in \mc{C}(\E,\cF)$. Then, when $\cF'$ runs over $\mathfrak{F}_0$, we have:
\begin{itemize}
\item if $\cF\subset \cF ^{'} \Rightarrow d_\lambda ^{\E, \cF}(X)\leq d_\lambda ^{\E, \cF^{'}}(X)$;
\item if $\cF^{'}\subset \cF $ and ${\mc R}(X_\E -\lambda I_{\E})\subset \cF{'} \Rightarrow d_\lambda ^{\E, \cF}(X)\geq d_\lambda ^{\E, \cF^{'}}(X)$;
\item if $\cF$ and $\cF'$ are not comparable, then there is no {\em a priori} relation between the corresponding defect numbers.
\end{itemize}
On the other hand, if $\E,\cF$ are fixed and $\E'$ runs over $\mathfrak{F}_0$,
\begin{itemize}
\item if $\E'\subset\E \Rightarrow d_\lambda ^{\E, \cF}(X)\leq d_\lambda ^{\E^{'}, \cF}(X)$;
\item if $\E\subset{\E^{'}}$ and  ${\mc R}(X_{\E'}-\lambda I_{\E'})\subset \cF \Rightarrow d_\lambda ^{\E, \cF}(X)\geq d_\lambda ^{\E^{'}, \cF} (X)$;
\item if $\E$ and $\E'$ are not comparable, then there is no {\em a priori} relation between the corresponding defect numbers.

\end{itemize}
\begin{prop} \label{propos} Let $\E, \cF \in \gF_0$, $X\in{\mc C}(\E, \cF)$ and $\lambda\in\CN$.\\
(i) $\lambda\in\pi^{\E, \cF}(X)$  if and only if $(X_\E -\lambda I_\E)$ has a bounded inverse $(X_\E -\lambda I_\E)^{-1}$ defined on ${\mc R}(X_\E -\lambda I_\E)\subseteq\cF$.\\
(ii)
${\mc R}( X_\E-\lambda I_\E) = \overline{{\mc R}(X-\lambda I)}$, for each $\lambda\in \pi_{\E, \cF} (X)$.\\
(iii) If $\lambda\in \pi^{\E,\cF}(X)$, then ${\mc R}(X-\lambda I)$ is closed in $\cF$.
\end{prop}
\begin{proof} (i) follows easily from \eqref{eq_oneone}.\\
(ii) Let $\eta$ be in the closure of ${\mc R}(X-\lambda I)$ in $\cF$. Then there is a sequence $(\xi_{n})_{n\in {\mb N}}$ of vectors $\xi _n\in \D$ such that $\eta_n = (X-\lambda I)\xi_n \rightarrow \eta $ in $\cF$. We have
$$\|\xi_n - \xi_m\|_\E\leq c_\lambda ^{-1} \|(X-\lambda I) (\xi_n - \xi_m)\|_\cF = \|\eta_n - \eta_m\|_\cF  .$$
Hence $(\xi_n)$ is a Cauchy sequence in $\E$ and so it is convergent. Let $\xi=\lim_{n\to\infty} \xi_n$. Then $\ X \xi_n = \eta_n + \lambda \xi_n \to \eta + \lambda \xi $. By the definition of $X_\E$ it follows that $X_\E \xi = \eta + \lambda \xi$, so that $\eta= (X_\E-\lambda I_\E)\xi\in {\mc R}(X_\E-\lambda I_\E)$. Thus $\overline{{\mc R}(X-\lambda I)} \subseteq {\mc R}(X_\E-\lambda I_\E)$. Using once more the definition of $X_\E$, it follows easily that ${\mc R}(X_\E-\lambda I_\E)\subseteq \overline{{\mc R}(X-\lambda I)}$ and, thus, the equality holds.
\\
(iii): It follows from (ii).
\end{proof}

\begin{prop} The set of the regular points of $X\in \LDD$ is an open subset of $\CN$.

\end{prop}
\begin{proof} Let $\lambda _0 \in \pi (X)$. Then there exist $\E,\cF \in \gF_0$, $\E \subseteq \cF$, such that $\lambda_0 \in \pi^{\E, \cF} (X)$. Let $\lambda \in \CN$ and suppose that $ |\lambda - \lambda _0| < c_{\lambda _0}$. Then for $\xi\in\D$, we have, on one hand,
$$\|(X-\lambda I)\xi\|_\cF \geq  \|(X-\lambda_0 I)\xi\|_\cF - |\lambda - \lambda_0| \|\xi\|_\E \geq (c_{\lambda_0} - |\lambda - \lambda_0| ) \|\xi\|_\E. $$
On the other hand
\begin{align*}
\|(X-\lambda I)\xi\|_\cF &\leq \|(X-\lambda_0 I)\xi\|_\cF +|\lambda - \lambda_0|\|\xi\|_\cF \\
&\leq \|(X-\lambda_0 I)\xi\|_\cF + c_{\lambda_0}\|\xi\|_\E \leq (d_{\lambda_0}+c_{\lambda_0})\|\xi\|_\E.
\end{align*}
Hence $\lambda \in \pi^{\E, \cF} (X)$. This in turn implies that $\pi (X)$ is open.
\end{proof}
With minor modifications of a classical argument (see, e.g. \cite[Proposition 2.4]{schmu2}) one obtains the following
\begin{prop}
Let $X\in \mc{C}(\E, \cF)$. Then the defect number $d_\lambda ^{\E, \cF}(X)$ is constant on each connected component of the open set $\pi^{\E,\cF} (X)$.
\end{prop}

Since the map $X \in {\mc C}(\E,\cF)\to X_\E \in  {\mc B}(\E, \cF)$ is a topological isomorphism, it can be more convenient to study certain properties in ${\mc B}(\E, \cF)$ rather than in ${\mc C}(\E,\cF)$. We begin with the notion of resolvent in ${\mc B}(\E, \cF)$.

\bedefi  Let $\E,\cF \in \gF_0$ with $\E\subseteq \cF$.
For $A \in {\mc B}(\E, \cF)$ we define
$$ \varrho_{\E,\cF} (A):=\{\lambda \in {\mb C}: (A-\lambda I_\E)^{-1} \mbox{ exists in }{\mc B}(\cF, \E) \}.$$
If $\lambda \in \varrho_{\E,\cF} (A)$ we put $R_\lambda^{\E,\cF}(A):= (A-\lambda I_\E)^{-1}.$
\findefi

Let $X \in \LDD$ and suppose that $X \in \mathcal{C}(\E, \cF)$, with $\E,\cF \in \gF_0$ with $\E\subseteq \cF$. Then we put $\varrho_{\E,\cF}(X):=\varrho_{\E,\cF}(X_\E)$ and $R_\lambda^{\E,\cF}(X):= (X_\E-\lambda I_\E)^{-1}$, if $\lambda \in \varrho_{\E,\cF}(X)$. It is clear that $\lambda\in \varrho_{\E, \cF}(X)$
if and only if $\lambda\in\pi^{\E, \cF}(X)$  and $(X_\E-\lambda I_\E)$ is surjective in $\cF$.
For consistency of notations we put $\varrho_{\E,\cF}(X)=\emptyset$ if  $X \not\in \mathcal{C}(\E, \cF)$. With this convention, one has
\begin{equation}\label{eq_unionresolvents} \varrho^{\gF_0} (X)= \bigcup_{\E,\cF \in \gF_0} \varrho_{\E,\cF}(X).
\end{equation}

\berem \label{rem_316}If $\gF_0$ is stable under duality (i.e., $\E \in \gF_0$ if and only if $\E^\times \in  \gF_0$), then by Remark \ref{rem_inverses} it follows that $\lambda \in \varrho_{\E,\cF}(X)$ if and only if $\overline{\lambda} \in \varrho_{\cF^\times,\E^\times }(X^\dag)$. This, in turn, implies that $\varrho^{\gF_0}(X)= \overline{\varrho^{\gF_0}(X^\dag)}$.\enrem
\begin{thm}Let $\E,\cF \in \gF_0$. The set $G({\mc B}(\E, \cF))$ of all invertible elements of ${\mc B}(\E, \cF)$ is open.
\end{thm}
\begin{proof}
Let $A\in G({\mc B}(\E, \cF))$ and $B\in {\mc B}(\E, \cF)$.
We can write $A+B=A(I_\E + A^{-1}B )$ and if we choose $B$ such that  $\|A^{-1}B\|_{\E,\E}< 1$, then  $I_\E + A^{-1}B\in G({\mc B}(\E, \E))$, since ${\mc B}(\E, \E)$ is a Banach algebra, and then $A+B\in G({\mc B}(\E, \cF))$.
\end{proof}
\begin{thm} Let $\E,\cF \in \gF_0$. The map $A \in G({\mc B}(\E, \cF))\to A^{-1} \in {\mc B}(\cF, \E)$ is continuous.
\end{thm}
\begin{proof}
Let $A$, $B\in G({\mc B}(\E, \cF))$, then:
\begin{eqnarray*}\|A^{-1}-B^{-1}\|_{\cF,\E}&=&\|B^{-1}(B-A)A^{-1}\|_{\cF,\E}\\ &=&\|(B^{-1}-A^{-1})(B-A)A^{-1} + A^{-1}(B-A)A^{-1}\|_{\cF,\E}\\&\leq&\|B^{-1}-A^{-1}\|_{\cF,\E}\|B-A\|_{\E,\cF}\|A^{-1} \|_{\cF,\E}\\ &+& \|B-A\|_{\E,\cF} \|A^{-1}\|_{\cF,\E}^2.\end{eqnarray*}
If we take $\|B-A\|_{\E,\cF}$ such that $(1-\|B-A\|_{\E,\cF} \|A^{-1}\|_{\cF,\E})\geq \frac{1}{2}$, then $\|A^{-1}-B^{-1}\|_{\cF,\E}$ can be made arbitrarily small.
\end{proof}

\berem \label{rem_pow}
Let $A,B\in\mathcal{B}(\cF,\E)$, with $\E\subseteq\cF$. Let us define $A_0:=A\upharpoonright_{\E}$. Then $A_0\in {\mc B}(\E, \E)$.  The product $A\cdot B$ is defined by $A\cdot B\eta = A_0(B\eta)$, for every $\eta \in \cF$ and $$ \|A_0B\|_{\cF,\E }\|\leq \|A_0\|_{\E,\E }\|B\|_{\cF,\E }.$$

In particular, if $A\in\mathcal{B}(\cF,\E)$, with $\E\subseteq\cF$, the $n$-th power $ A^{(n)}$ of $A$ is defined by
$$ A^{(2)}:= A_0A, \quad\mbox{and}\quad A^{(n)}= A_0 A^{(n-1)}$$ and one has
\begin{equation}\label{eq_powers} \|A^{(n)}\|_{\cF,\E } \leq \|A_0\|_{\E,\E }^{n-1} \|A\|_{\cF,\E }.\end{equation}

\enrem

\begin{lemma} \label{lemma_residentities} Let $A,B\in\mathcal{B}(\E,\cF)$, with $\E\subseteq\cF$. The following resolvent identities hold:
\begin{align*} &R_\lambda^{\E,\cF} (A) - R_\lambda^{\E,\cF}  (B) =R_\lambda (A)^{\E,\cF}  (A-B) R_\lambda^{\E,\cF}  (B), \hspace{0,3cm}\forall \lambda \in \varrho_{\E,\cF} (A)\cap \varrho_{\E,\cF} (B);\\
&R_\lambda^{\E,\cF} (A) - R_{\lambda _0}^{\E,\cF} (A) = (\lambda - \lambda _0)R_\lambda^{\E,\cF} (A) R_{\lambda _0}^{\E,\cF} (A), \hspace{0,3cm}\forall \lambda , \lambda_0 \in \varrho_{\E,\cF}(A).\end{align*}
\end{lemma}

\begin{thm}\label{thm_resprop}Let $\E,\cF \in \gF_0$ and $A \in {\mc B}(\E, \cF)$. The following statements hold:
\begin{itemize}\item[(i)] $\varrho_{\E,\cF} (A)$ is open.
\item[(ii)] the function $\lambda \in \varrho_{\E,\cF} (A) \to  (A-\lambda I_\E)^{-1}\in {\mc B}(\cF, \E)$ is analytic on every connected component of  $\varrho_{\E,\cF} (A)$.\end{itemize}
\end{thm}
\begin{proof}
 (i): Put $h(\lambda)= A-\lambda I_\E$. Then $\varrho_{\E,\cF} (A)= h^{-1}(G({\mc B}(\cF, \E))$ and so it is open, since $h$ is clearly continuous.

(ii):

If $\lambda \to \mu$, $(A-\lambda I_\E)^{-1} \to (A-\mu I_\E)^{-1}$, due to the continuity of the inversion. Since $\E \subset \cF$ and $I_\E: \E\to \E\subset \cF$ is continuous, we have $\|B\restr{\E}\|_{\E,\E}\leq \|B\|_{\cF,\E}$, for every $B \in {\mc B}(\cF, \E)$. Thus,
\mbox{$\|(A-\lambda I_\E)^{-1}\restr{\E} - (A-\mu I_\E)^{-1}\restr{\E}\|_{\E,\E}\to 0$} as $\lambda \to \mu$.
Hence,
\begin{align*}
\|(A-\lambda I_\E)^{-1}\restr{\E}&(A-\mu I_\E)^{-1} - (A-\mu I_\E)^{-1}\restr{\E}(A-\mu I_\E)^{-1}\|_{\cF,\E}\\
&\leq \|(A-\lambda I_\E)^{-1}\restr{\E} - (A-\mu I_\E)^{-1}\restr{\E}\|_{\E,\E} \|(A-\mu I_\E)^{-1}\|_{\cF, \E}\to 0,
\end{align*}
as $\lambda \to \mu$.
Finally, we get
\begin{align*}\lim_{\lambda \to \mu} \frac{(A-\lambda I_\E)^{-1} - (A-\mu I_\E)^{-1}}{\lambda-\mu}&= \lim_{\lambda \to \mu}( (A-\lambda I_\E)^{-1}\restr{\E}(A-\mu I_\E)^{-1})\\&=(A-\mu I_\E)^{-1}\restr{\E}(A-\mu I_\E)^{-1},\end{align*}
in the norm of ${\mc B}(\cF,\E)$.
\end{proof}

By (i) of Theorem \ref{thm_resprop} and by \eqref{eq_unionresolvents}, we get
\begin{prop} Let $X \in \LDD$ and $\gF_0$ a family of interspaces. Then, $\varrho^{\gF_0}(X)$ is open.
\end{prop}

Adopting the notations of Remark \ref{rem_pow} we have
\begin{prop}
Let $X\in \LDD$, $\E\subset\cF$ and $\lambda_0 \in \varrho_{\E,\cF}(X)$. Then there exists $\delta>0$ such that, for every $\lambda \in {\mb C}$ with $|\lambda-\lambda_0|<\delta$, $\lambda \in \varrho_{\E,\cF}(X)$ and
$$ R_{\lambda}^{\E,\cF} (X) = \sum_{n=0} ^{+\infty} (\lambda - \lambda_0)^n R_{\lambda_0}^{\E,\cF} (X)^{(n+1)} ,$$
where the series converges in the operator norm of $\mathcal{B}(\cF,\E)$.
\end{prop}
\begin{proof}
We already know that $\varrho_{\E,\cF} (X)$ is an open set; thus there exists $r>0$ such that the disk $|\lambda- \lambda_0|<r$ is contained in $\varrho_{\E,\cF} (X)$. Let $\delta=\min\left\{ r, \left\|\left(R_{\lambda_0}^{\E,\cF} (X)\right)_0\right\|^{-1}_{\E,\E}\right\}$,  hence $\lambda\in \varrho_{\E,\cF} (X)$ and

\begin{eqnarray*}& & \left\| \sum_{n=k+1} ^{k+p} (\lambda - \lambda_0)^n R_{\lambda_0}^{\E,\cF} (X)^{(n+1)}\right\|_{\cF,\E}= \left\|\sum_{n=k+1} ^{k+p} (\lambda - \lambda_0)^n \left(R_{\lambda_0}^{\E,\cF} (X)\right)_0^{(n)} R_{\lambda_0}^{\E,\cF} (X)\right\|_{\cF,\E}\\
& \leq& \left\|\sum_{n=k+1} ^{k+p} (\lambda - \lambda_0)^n \left(R_{\lambda_0}^{\E,\cF} (X)\right)_0^{(n)}\right\|_{\E,\E}\left\| R_{\lambda_0}^{\E,\cF} (X)\right\|_{\cF,\E}
\end{eqnarray*}

Since $\left\|(\lambda - \lambda_0) \left(R_{\lambda_0}^{\E,\cF} (X)\right)_0\right\|_{\E,\E}<1$, the series
$$ \sum_{n=0} ^{+\infty} (\lambda - \lambda_0)^n R_{\lambda_0}^{\E,\cF} (X)^{(n+1)}$$
converges in the operator norm of $\mathcal{B}(\cF,\E)$.

Using the second identity in Lemma \ref{lemma_residentities}, we finally obtain, in standard way,
$$ R_{\lambda}^{\E,\cF} (X) = \sum_{n=0} ^{+\infty} (\lambda - \lambda_0)^n R_{\lambda_0}^{\E,\cF} (X)^{(n+1)}.$$
\end{proof}

\begin{lemma} Let $\lambda \in \varrho_{\E,\cF}(X)$, $\E, \cF \in \gF_0$, $\E\subseteq \cF$. Then
\begin{itemize}
\item[(i)] $\lambda \in \sigma_{\E,\cF'}(X)$, for every $ \cF' \in \gF_0$, $\cF' \supsetneqq \cF$;
\item[(ii)] if $ \E \varsubsetneqq \E' \subseteq \cF$ and $X\in \mc{C}(\E', \cF)$, then $\lambda$ is an eigenvalue of $X_{\E'}$ and, hence, $\lambda \in \sigma_{\E', \cF}(X)$.
\end{itemize}
\end{lemma}
\begin{proof} (i) is straightforward. As for (ii) the proof is similar to that given in Remark \ref{rem_globinverse}.
\end{proof}

Recalling that, if $B \in {\mc B}(\E,\cF)$, then $B\upharpoonright_\D\in \LDD$, it is natural to give the following
\bedefi \label{def_B}Let $\E, \E', \cF, \cF'\in \gF_0$ and $B \in {\mc B}(\E,\cF)$, $C \in {\mc B}(\E',\cF')$. We say that $B$ and $C$ are equivalent, and write $B\equiv C$, if $B\upharpoonright_\D=C\upharpoonright_\D$.

\findefi

For fixed $\E, \cF\in \gF_0$,with $\E\subseteq \cF$, $X\in\mc{C}(\E, \cF)$ the function
$$ \lambda \in \varrho_{\E,\cF} (X_\E)\to  (X_\E -\lambda I_\E)^{-1} \in {\mc B}(\cF,\E)$$
is analytic on $\varrho_{\E,\cF} (X_\E)$. Thus, if we define $(X -\lambda I)^{-1}_{(\E,\cF)}:= (X_\E -\lambda I_\E)^{-1} \upharpoonright_\D$, for every $\lambda \in \varrho_{\E,\cF} (X_\E)$, $(X -\lambda I)^{-1}_{(\E,\cF)}\in \LDD$.

 Moreover, the function
$$\lambda \in \varrho_{\E,\cF} (X_\E)\to (X -\lambda I)^{-1}_{(\E,\cF)}\in \LDD$$
is analytic if $\LDD$ is endowed with the inductive topology $\tau_{\rm in}$ defined by the family of Banach spaces $\{{\mc B}(\E, \cF);\,  \E, \cF\in \gF_0\}.$

Let us fix $\E\in \gF_0$. If $\lambda \in \varrho_{\E,\cF}(X)$, for some $\cF\in \gF_0$, $\E \subseteq \cF$, then this $\cF$ is unique.
Hence, the function $f_\E$ defined by
\begin{equation}\label{eq_branch} f_\E(\lambda) = (X -\lambda I)^{-1}_{(\E,\cF)} , \quad \mbox{if } \lambda \in \varrho_{\E,\cF}(X)\end{equation}
is a single valued function, analytic on every connected component of every open set $\varrho_{\E,\cF}(X)$.

\medskip
More in general, when $\E,\cF$ run over $\gF_0$ we get a whole family of resolvent functions and several different situations are possible.

If $\lambda_0\in \varrho^{\gF_0}(X)$, then $\lambda_0\in \varrho_{\E,\cF} (X_\E)$ for some $\E,\cF\in \gF_0$, $\E\subseteq \cF$. Then there exists an open disk $D_r=\{\lambda: |\lambda-\lambda_0|<r\} \subset \varrho_{\E,\cF} (X_\E) $ and the function $\lambda\in D_r \to (X_\E -\lambda I_\E)^{-1}$ is analytic on $D_r$. But it may happen that $\lambda_0\in \varrho_{\E',\cF'} (X_{\E'})$ for another couple ${\E}',{\cF'}\in \gF_0$, $\E'\subseteq \cF'$. The corresponding resolvent function $\lambda\to (X_{\E'} -\lambda I_{\E'})^{-1}$ is analytic in some open disk $D_{r'}$, but $(X_\E -\lambda I_\E)^{-1}$ and $(X_{\E' } -\lambda I_{\E'})^{-1}$ need not be equivalent on $D_r\cap D_{r'}$.
\begin{lemma}\label{lemma_320}
Let $\lambda_0 \in \varrho_{\E,\cF} (X_\E) \cap \varrho_{\E',\cF'} (X_{\E'})$ for some $\E, \E',\cF, \cF'\in \gF_0$. Let us assume that $\E \subset \E'$.
The corresponding resolvent functions are equivalent on some open neighborhood of $\lambda_0$ and they are direct analytic continuations of each other.
\end{lemma}
\begin{proof}
If $\E \subset \E'$, then $\cF \subset \cF'$ and $(X_{\E' } -\lambda I_{\E'})^{-1}$ is an extension of $(X_\E -\lambda I_\E)^{-1}$. Indeed, since $X_\E -\lambda I_\E \subseteq X_{\E' } -\lambda I_{\E'}$, their inverses coincide on $(X_\E -\lambda I_\E)\E=\cF$. This implies that  $(X_{\E' } -\lambda I_{\E'})^{-1}\equiv (X_\E -\lambda I_\E)^{-1}$ in the sense of Definition \ref{def_B} and the corresponding resolvent functions are direct analytic continuation one of the other.  \end{proof}
\begin{prop} Let $\lambda_0 \in \varrho_{\E,\cF} (X_\E) \cap \varrho_{\E',\cF'} (X_{\E'})$ for some $\E, \E',\cF, \cF'\in \gF_0$. If there exist ${\mc G}$ and ${\mc G}'$ in $\gF_0$ such that ${\mc G}\subseteq \E \cap \E'$ and $\lambda_0 \in \varrho_{{\mc G},{\mc G}'}(X)$,
then the functions $\lambda\to (X_{\E} -\lambda I_{\E})^{-1}$ and $\lambda\to (X_{\E'} -\lambda I_{\E'})^{-1}$ are analytic continuations of each other in some open connected set containing $\lambda_0$.\end{prop}
\begin{proof} One can apply the reasoning of Lemma \ref{lemma_320} to $(X_\E -\lambda_0 I_\E)^{-1}$ and $(X_{\mc G} -\lambda_0 I_{\mc G})^{-1}$ or to $(X_{\E'} -\lambda_0 I_{\E'})^{-1}$ and $(X_{\mc G} -\lambda_0 I_{\mc G})^{-1}$. Then, the functions $\lambda\to (X_{\E} -\lambda I_{\E})^{-1}$ and $\lambda \to (X_{\mc G} -\lambda I_{\mc G})^{-1}$ are equivalent in some open disk $D_r$ and they are direct analytic continuation one of the other. The same happens of course for the functions $\lambda\to (X_{\E'} -\lambda I_{\E'})^{-1}$ and $\lambda \to (X_{\mc G} -\lambda I_{\mc G})^{-1}$. Thus the functions indexed by $\E$ and $\E'$ are analytic continuations of each other (but they need not be direct analytic continuation one of the other).
 \end{proof}

\berem Let us finally examine the general situation. Let $\lambda_0 \in \varrho_{\E,\cF} (X_\E) \cap \varrho_{\E',\cF'} (X_{\E'})$ for some $\E, \E',\cF, \cF'\in \gF_0$. Since $\D$ is dense in $\E\cap\E'$, for the Mackey topology, then $(X_\E -\lambda_0 I_\E)\upharpoonright _{(\E\cap {\E}')}(\E\cap \E') = (X_{\E'} -\lambda_0 I_{\E'})\upharpoonright _{(\E\cap \E')}(\E\cap \E') = \L\subseteq \cF\cap\cF'$. The equality $\L=\cF\cap\cF'$ does not hold in general.
$(X_\E -\lambda_0 I_\E)^{-1}$ is well-defined in $\cF$ and then also in $\L$. The same holds for $(X_{\E'} -\lambda_0 I_{\E'})^{-1}$ but these operators need not be equivalent.
 \enrem

If $\lambda_0 \in \varrho^{\gF_0}(X)$, we denote by $(\mathbf{X } -\lambda_0 \mathbf{ I})^{-1}$ the collection of all resolvent operators corresponding to $\lambda_0$ and by $\lambda \in \varrho^{\gF_0}(X) \to (\mathbf{X -\lambda I})^{-1}$ the multivalued resolvent function described above. For each fixed $\E \in \gF_0$, the function $f_\E$, defined in \eqref{eq_branch}, can be viewed as a single valued branch of $ \lambda \to(\mathbf{X -\lambda I})^{-1}$. In Example \ref{ex. 42} we shall see a concrete realization of this situation.
\section{Hilbert space operators} \label{sect_Hilbert}
Let $X \in \LDD$ and assume that both $X$ and $X^\dag$ map $\D$ into $\H$. Then $X$ can be viewed as a closable operator in $\H$. Let $\overline{X}$ be its closure and $\varrho_\H(\overline{X})$ its usual resolvent set. We denote by $\H_X$ the Hilbert space obtained by endowing $D(\overline{X})$ with the graph norm. If $\H_X, \H\in \gF_0$ then $X \in {\mc C}(\H_X,\H)$ and $\varrho_{\H_X,\H}=\varrho_\H(\overline{X})$; so that $\sigma^{\gF_0}(X)\subseteq \sigma_\H(\overline{X})$ (this implies, in particular that, if $X$ is bounded, $\sigma^{\gF_0}(X)$ is compact). As we will see below, however, $\sigma_\H(\overline{X})$ and $\sigma^{\gF_0}(X)$ need not coincide.
\subsection{Rigged Hilbert spaces generated by symmetric operators} \label{subsect_symmetricop}
Let $A$ be a self-adjoint operator in Hilbert space $\H$. The space $\D= \D^\infty (A)$, endowed with its natural topology $t_A$, defined by the seminorms $p_n(\xi)= \|A^n \xi\|$, $n \in {\mb N}$, generates in canonical way a RHS, with $\D$ a Fr\'echet space. For every $n \in {\mb N}$ we denote by $\H_n$ the Hilbert space obtained by endowing $D(A^n)$ with its graph norm \mbox{$\|\cdot\|_n:=\| (I + A^{2n})^{1/2} \cdot \|$}and by $\H_{-n}$ the space obtained by completing $\H$ with respect to the norm $\|\cdot \|_{-n}:= \|(I + A^{2n})^{-1/2}\cdot\|$. Put $\H_0:=\H$. Then, the family of spaces $ \{\H_n; \, n\in {\mb Z}\}$ is totally ordered; namely,
$$ \cdots \H_{n+1} \subset \H_n \subset \cdots \subset \H=\H_0 \subset \H_{-n} \subset \H_{-n-1} \cdots $$

Let us put $S = A\!\upharpoonright\! _\D$ and take $\gF _0 = \{ \h _n ; n\in {\mb Z}\}$.
{The operator $A$ (or its extension by duality denoted by the same symbol) maps $\h _n$ in $\h _{n-1}$, $\forall n\in {\mb Z}$ continuously; hence $S\in C (\h _n, \h _{n-1})$, for every $n\in {\mb Z}$.}
Let us denote by $\varrho _\h (A)$ the usual resolvent of $A$.
 For shortness, we will put $\varrho _{n,m}(S):=\varrho _{\h _n , \h _m}(S) $.
 \begin{prop}\label{prop_selfadj} Let $A$ be a self-adjoint operator, $\D$ and $\gF_0$ as above. Then $\varrho ^{\gF _0} (S) = \varrho _\h (A)$.
 \end{prop}
\begin{proof}
 Let $\lambda \in \varrho _{n,n-1} (S)$, { $n\geq 0$. Then, } $(A-\lambda I)^{-1} \in {\mc B}(\h _{n-1}, \h _n)$, indeed
$$ \| (A-\lambda I)^{-1} \xi \| _n \leq C\| \xi\| _{n-1},\quad \forall \xi \in \h_{n-1}$$
i.e.
$$ \| R_n(A-\lambda I)^{-1} \xi \| \leq C\| R_{n-1}\xi\|, \quad \forall \xi \in \h_{n-1},$$
where $R_n = (I+A^{2n})^{\frac{1}{2}}$.
Then
$$ \| (A-\lambda I)^{-1} R_n \xi \| \leq C\| R_{n-1}\xi\| \leq C \| R_n\xi\|, \quad \forall \xi \in \h _n .$$
This implies that $(A-\lambda I)^{-1}$ is bounded w.r. to the norm of $\h$ on the subspace $R_n \h_n$, and, since $R_n$ is invertible with bounded inverse, it follows that $\lambda$ belongs to the usual resolvent, $\varrho _\h (A)$, of $A$.
{Let $\mu \in \varrho _{-n+1,-n} (S)$, $n\geq 0$. Then, by Remark \ref{rem_316}, $\overline{\mu}\in \varrho _{n,n-1}(S)$. The first part of the proof then shows that $\overline{\mu}\in \varrho _\h (A)$. This, in turn, implies that $\mu \in \varrho _\h (A)$.}

{Let now $\lambda\in \varrho _\h (A)$; then $(A-\lambda I)^{-1} \in \B(\h)$. Now we want to prove that $\lambda \in \varrho _{n,n-1}(S)$, $\forall n\geq 1$.
Since $(A-\lambda I)^{-1}$ maps $\h_n$ into $ \h_{n-1}$ and it is is clearly injective, we only need to prove the surjectivity.
Let $\eta\in \h_{n-1}\subseteq \h$ and $\xi\in D(A)$ be such that $(A-\lambda I)\xi = \eta$, then necessarily $\xi = (A-\lambda I)^{-1} \eta\in \h_n$. }

Finally, it is easy to see that $\varrho_{n,m}(S)=\emptyset$ if $m\neq n-1$.

In conclusion, $\varrho ^{\gF _0} (S) = \varrho _\h (A)$.
\end{proof}

The situation becomes more involved if $S$ is a symmetric operator possessing many self-adjoint extensions. In this case in fact, we have to deal with a true multivalued resolvent function.
\beex \label{ex. 42}
Let $S$ be a closed symmetric operator with equal and finite defect indices. Again we put $$\D^\infty (S) = \bigcap _{n\geq 0} \D (S^n)$$
and, also in this case, $\D^\infty (S)$ is dense in $\h$ \cite[Prop. 1.6.1]{schmu}.
If $S^{'}$ is a self-adjoint extension of $S$, we clearly have $$\D(S^n)\subset \D({S^{'}}^n),\quad \forall n\geq 1$$
and then
$$ \D^\infty (S) \subset \D^\infty (S^{'})  .$$
Let assume that $S$ has a family $\{S_\alpha\}_{\alpha\in I}$ of self-adjoint extensions.
We put $\h_{\alpha , n} = \D(S_\alpha ^n)$ endowed with the graph norm as before and consider
$$  \gF _0 = \{ \h _{\alpha , n}; \alpha\in I, n\in {\mb N}\}$$
Then $S\in \mc{C}(\h _{\alpha , n}, \h _{\beta , m} )$ if and only if $\alpha = \beta$ and $m\leq n-1$.
By the previous result, it follows that
$$ \varrho_{\h_{\alpha , n}, \h_{\alpha , n-1}} (S) = \varrho _\h (S_\alpha) .$$
Hence $\varrho ^{\gF _0}(S) = \cup_{\alpha\in I}\hspace{0,2cm} \varrho _\h (S_\alpha)$.

Let $S_\alpha$ and $S_\beta$ be two different self-adjoint extensions of $S$; then the essential spectra $\sigma _{ess} (S_\alpha)$ and $\sigma _{ess} (S_\beta)$ are equal \cite[Theorem 8.18]{weidman}, while the point spectra $\sigma _p (S_\alpha)$ and $\sigma _p (S_\beta)$ are different, in general.

A well known concrete example is provided by the differential operator on an interval of the real line.
Let us consider, in fact, $$D(S^*)=\{f\in L^2([0,1]):f(x)=f(0)+\int_0^xg(t)dt; g\in L^2([0,1])\},$$
$$D(S)=\{f\in D(S^*):f(0)=f(1)=0\},$$
$$D(S_\alpha)=\{f\in D(S^*):f(1)=\alpha f(0),\hspace{0,2cm} \alpha\in \mathbb{C} \mbox{ with } |\alpha|=1\}$$ and $S^*f:=-ig.$

Then $S$ is closed and symmetric but not self-adjoint, $S_\alpha$ is a self-adjoint extension of $S$ for every $\alpha\in \mathbb{C}$ with $|\alpha|=1$. The point spectrum of $S$ is empty and its whole spectrum is $\sigma(S)=\mathbb{C}$; as for the $S_\alpha$'s, one has $\sigma_p(S_\alpha)=\{arg(\alpha)+2k\pi, k\in\mathbb{Z}\}=\sigma_\H(S_\alpha)$. Hence  $$\varrho ^{\gF _0} (S)=\bigcup_{\alpha:\hspace{0,09cm} |\alpha|=1}\varrho_\H(S_\alpha)=\bigcup_{\alpha: \hspace{0,09cm}|\alpha|=1}(\mathbb{C}\setminus\{arg(\alpha)+2k\pi, k\in\mathbb{Z}\})=\mathbb{C}.$$

It is worth remarking that the result $\varrho ^{\gF _0} (S)={\mb C}$ does not change if we take as $I$ a proper subset $J$ of the unit circle $\{\alpha \in {\mb C}; |\alpha|=1\}$ consisting of at least two different (modulo $2\pi$) points.

Now we consider the ``global'' resolvent multivalued function
$$\lambda\in\varrho^{\gF_0}(S)\to(\textbf{S}-\lambda\textbf{ I})^{-1}=\{(S_\alpha-\lambda I)^{-1}\}.$$
Let us  first introduce some notation. We recall that the operator  $S$ has defect indices $(1,1)$. For every $\lambda\in \mathbb{C}\setminus\mathbb{R}$, the subspace $M_\lambda$ of solutions of the equations $$S^*\Phi_\lambda=\lambda\Phi_\lambda,$$ has dimension $1$.\\ It is easily seen that $\Phi_\lambda(x)=Ke^{i\lambda x}$.\\

The Krein formula \cite{Akhiezer}, allows us to compute:

\begin{equation}\label{1}
    ((S_\alpha-\lambda I)^{-1}-(S_\beta-\lambda I)^{-1})g=\mu(\lambda)\ip{g}{\Phi_{\overline{\lambda}}}\Phi_{\lambda},
\end{equation}
where $\mu(\lambda)\neq0$ and the functions $\lambda\to \mu(\lambda)$ and $\lambda\to \Phi_\lambda$ are analytic in $\varrho_\H(S_\alpha)\bigcap\varrho_\H(S_\beta)$.\\
This formula shows that, in general, $(S_\alpha-\lambda I)^{-1}$ and $(S_\beta-\lambda I)^{-1}$ are not analytic continuations of each other, since the r.h.s. in \eqref{1} is zero if and only if $S_\alpha=S_\beta$. More precisely, in our case

\begin{align*}(((S_\alpha-\lambda I)^{-1}-(S_\beta &-\lambda I)^{-1})g)(x)\\ &=\left(\frac{1}{\alpha-e^{i\lambda}}\,-\,\frac{1}{\beta-e^{i\lambda}}\right)ie^{i(x+1)\lambda}\int_0^1g(\tau)e^{-i\lambda\tau} d\tau\end{align*} which vanishes if and only if $\alpha=\beta$.
\enex

\beex Let us consider the operator $H_0=-\frac{d^2}{dx^2} $. Our aim here is to provide a  family of
intermediate spaces for $H_0$ and find the corresponding
resolvent.  Of course, there are several possible domains $\D$ such that $H_0$ (or, better, its restriction to $\D$) can be considered as an element of $\LDD$. Let us examine shortly two cases.
\begin{enumerate} \item First, we take $\D= \S:=\S({\mb R})$. In this case, $H_0$ is the so-called free Hamiltonian of quantum mechanics and the simplest Schr\"odinger operator. It is easily seen that $H_0 \in \gl(\S,\S^\times)$ (more precisely, $H_0\in \Lc^\dagger(\S)$) and it is essentially self-adjoint on $\S$; its closure, $H_1:=\overline{H_0}$, is defined on the Sobolev space $W^{2,2}({\mb R})$.
As discussed at the beginning of Section 4.1, a natural choice for the family $\gF _0$ consists in taking the scale of Hilbert spaces generated by $H_1$. This is, actually, a chain of Sobolev spaces; i.e., $\gF _0 = \{ W^{2m,2}(\mathbb{R})[\|\cdot \|_{m}] ; m\in {\mb Z}\}$ (as before, \mbox{$\|\cdot \|_{\pm n}:=\| (I +H^{2n})^{\pm 1/2} \cdot \|$}, with $n\in\mathbb{N}$).
Hence, by Proposition \ref{prop_selfadj}, $$\sigma ^{\gF _0} (H_0) = \sigma _{\H} (H_1) = \mathbb{R}^+ \cup \{0\}.$$
\item A second case of interest arises if we impose a boundary condition by taking, for instance,  $\D:=\S_y:= \{f \in \S: f(y)=0\}$, $y \in {\mb R}$, with the topology induced by $\S$. This domain is used when {\em perturbing} the free Hamiltonian with a $\delta$-interaction centered at $y$, \cite{albeverio}.
The domain of the closure $H_1$ of $H_0$ is $W_y^{2,2}(\mathbb{R})= \{f\in W^{2,2}(\mathbb{R}) ; f(y)=0\}$. As shown in \cite[Th. 3.1.1]{albeverio}, the operator $H_1$ is no longer self-adjoint; it has, in fact, defect indices (1,1) and, for each $\alpha\in \mathbb{R}$, it possesses a self-adjoint extension $H_\alpha$.
The domain of $H_\alpha$ is
$$ \hspace{1.2cm}D(H_\alpha)=\left\{ g\in W^{1,2}(\mb R) \cap W^{2,2}({\mb R}\setminus \{y\}):\, g'(y^+)-g'(y^-)=\alpha g(y)\right \}. $$

As for the spectrum, we have

 $$ \sigma_\H (H_\alpha)=\left\{\begin{array}{ll}\mathbb{R}^+ \cup \{0\} & \mbox{ if } \alpha \geq 0\\ \mathbb{R}^+ \cup \{-\frac{\alpha^2}{4},0\} & \mbox{ if } \alpha < 0,\end{array}\right. $$ since for $\alpha<0$, $-\frac{\alpha^2}{4}$ is an eigenvalue of $H_\alpha$.

Then, proceeding as in Example \ref{ex. 42}, we get that $$\varrho ^{\gF _0} (H)=\bigcup_{\alpha\hspace{0,09cm} \in \mathbb{R}}{\varrho}_{\H}(H_\alpha) = \mathbb{C}\setminus \{\mathbb{R}^+\cup \{0\}\}$$ where $  \gF _0 = \{ \H_{\alpha , n}; \alpha\in I, n\in {\mb N}\}$ (with $ \H_{\alpha , n} = \D(H_\alpha ^n)$ endowed with the graph norm, as before). Hence, also in this case, we get $$\sigma ^{\gF _0} (H) = \mathbb{R}^+ \cup \{0\}.$$\end{enumerate}

\enex

\subsection{Generalized eigenvalues and generalized eigenvectors}\label{sect_ geneigenvalue}
As announced in Section \ref{sect_resolvent}, we consider now the problem of the existence of a complete set of generalized eigenvalues of a self-adjoint operator $\H$ and give a slight improvement of Gelfand theorem on this subject.

Let $A$ be a self-adjoint operator with $A\geq I$ in Hilbert space $\H$. Let us assume that $A^{-1}$ is a Hilbert-Schmidt operator. Let us take, as in Section \ref{subsect_symmetricop}, $\D=\D^{\infty}(A)$ and $\gF_0=\{\H_n, n \in {\mb Z}\}$ the chain of Hilbert spaces generated by the powers of $A$
(in this case, the graph norm of $\H_n$ can be equivalently be taken as $\|\cdot\|_n=\|A^n\cdot\|$).
Since $\D$ is a Fr\'echet space, $\LBDD{}=\LDD$ and, for every $X \in \LDD$, there exists $n \in {\mb N}$ such that $X \in \mc{C}(\H_n, \H_{-n})$. Hence, we can define an operator $X_0$ in the following way
\begin{align*}
&D(X_{0,n})= \{\xi \in \H \cap \H_n: X\xi \in \H\},\\
&X_{0,n}\xi=X\xi, \; \xi \in D(X_0).
\end{align*}
In general, $D(X_0)$ is not dense in $\H$ and may reduce to the null subspace only.

\begin{thm} \label{thm_geigen}Let $X\in \LDD$ be symmetric, with $\D= \D^\infty(A)$ where $A$ is a self-adjoint operator with $A\geq I$, in Hilbert space $\H$, whose inverse $A^{-1}$ is Hilbert-Schmidt.
Assume that  $X \in \mc{C}(\H_n, \H_{-n})$, for some $n \in {\mb N}$, and that $X_{0,n}$ is densely defined and essentially selfadjoint.  Then $X$ has a complete set of generalized eigenvectors.
\end{thm}

\begin{proof} Let $\{E(\lambda)\}$ be the spectral family of $\overline{X_0}$ and $\zeta$ a unit vector in $\H$. Put $\sigma(\lambda)=\ip{E(\lambda)\zeta}{\zeta}$. Then $\sigma$ defines a measure on ${\mb R}$ and almost everywhere with respect to $\sigma$ there exists the derivative $$ \frac{dE(\lambda)\zeta}{d\sigma(\lambda)}=:\chi(\lambda).$$
As shown in \cite[Sect. 4.3, Theorem 1]{gelfand}, $\chi(\lambda)$ is a continuous linear functional on $\D$. This set of functionals is complete in the sense that, for every vector $\phi \in \M(\zeta)$, the closed subspace generated by $\{ E(\lambda)\zeta; \lambda \in {\mb R}\}$, one has \cite[Sect. 4.3, Theorem 2]{gelfand}
\begin{equation}\label{eq_geneigenvect} \phi= \int_{\mb R} \overline{ \ip{\chi(\lambda)}{\phi}}\chi(\lambda) d \sigma(\lambda) \;\mbox{ and }\;
\|\phi\|^2= \int_{\mb R} |{ \ip{\chi(\lambda)}{\phi}}\chi(\lambda)|^2 d \sigma(\lambda).\end{equation}
We want to prove that $\chi(\lambda) \in \H_{-n}$.
We denote by $\Delta$ the interval $[\alpha, \beta]$ containing the point $\lambda$ and by $E(\Delta)$ the operator $E(\beta)-E(\alpha)$. Assume that the interval $\Delta$ contracts to the point $\lambda$. Then, for every $\xi \in \D$, we have
\begin{align*} \ip{\chi(\lambda)}{\phi}&= \lim_\Delta \ip{\frac{E(\Delta)\zeta}{\sigma(\Delta)}}{\phi}\\ &= \lim_\Delta \ip{A^{-n}\frac{E(\Delta)\zeta}{\sigma(\Delta)}}{A^n\phi}.
\end{align*}
Now, we observe that $\|E(\Delta)\zeta/\sigma(\Delta)\|=1$; hence by the compactness of $A^{-1}$, there exists a subnet $\{E(\Delta')\zeta/\sigma(\Delta')\}$ such that $A^{-n}E(\Delta')\zeta/\sigma(\Delta')$ converges in $\H$. This implies that
$$ |\ip{\chi(\lambda)}{\phi}|\leq C \|A^n\phi\|= C\|\phi \|_n, \quad \phi \in \D.$$
Thus $\chi(\lambda)$ extends to a continuous conjugate linear functional on $\H_n$. We denote this extension by the same symbol. Since $X \in \mc{C}(\H_n, \H_{-n})$, denoting by $X_{(n)}$ the corresponding extension we get, for every $\phi \in \D$, in complete analogy to \cite[Sect. 5.2, Theorem 1]{gelfand},
\begin{align*}
\ip{X_{(n)}\chi(\lambda)}{\phi}&=\ip{\chi(\lambda)}{X\phi} \\
&= \lim_\Delta \ip{\frac{E(\Delta)\zeta}{\sigma(\Delta)}}{X\phi}\\
&= \lambda \ip{\chi(\lambda)}{\phi}.
\end{align*}
Hence $\chi(\lambda)$ is an $\gF_0$-generalized eigenvector.
The final step of the proof consists in decomposing $\H$ into an orthogonal sum of subspaces of the type $\M(\zeta_\alpha)$, as in \cite{gelfand}.
\end{proof}

\berem We recall a well known situation where Theorem \ref{thm_geigen} can be applied. This is the case where $X \in \mc{C}(\H_1, \H_{-1})$ and $[A_0, X] \in \mc{C}(\H_1, \H_{-1})$, (here $[A_0, X]$ denotes the commutator of $A_0:=A\upharpoonright_\D \in \LD$ and $X$ which is well-defined since $A_0$ has a continuous extension $\hat{A}_0$ to $\D^\times$ and, therefore, one can define
$[A_0, X]\xi = \hat{A}_0X\xi - XA_0\xi, \quad \xi \in \D$).

Then by the commutator theorem \cite[Sec. X.5]{reedsimon}, $X_0$ is densely defined and it is essentially self-adjoint on every core for $A$.\enrem

\berem Once a notion of spectrum is at hand, it is natural to pose the question as to whether other aspects of the beautiful spectral theory for  operators in Hilbert space extend to the  different environment we have considered here. Thus, one first asks oneself if a given operator $X=X\ad \in \LDD$ gives rise to a resolution of the identity in a possibly generalized sense. Some hints come from the previous discussion on generalized eigenvectors and eigenvalues we have done in this section. However, in the case considered here, for instance in Theorem \ref{thm_geigen}, we have, since from the very beginning, a  well behaved operator $X$ and both the resolution of the identity $\{\chi(\lambda); \lambda \in {\mb R}\}$ of the rigged Hilbert space obtained in \eqref{eq_geneigenvect} and the measure $\sigma$ are determined by the spectral family of the essentially self-adjoint operator $\overline{X_0}$.  In the general case, we conjecture that it is not possible to associate a resolution of the identity to a symmetric element of $\LDD$. These operators in fact can be so far from being Hilbert space operators (see, for instance the example in Section 5) as to make impossible the use of the powerful tools at our disposal in a Hilbert space theory. Thus, a first step in the direction of getting more results on the existence of a resolution of the identity should consist in considering operators with a {\em sufficiently large} restriction to the central Hilbert space $\H$.  We hope to consider this problem in a future paper.
\enrem

\berem In the language of Quantum Mechanics the {\em resonant spectrum} of the Hamiltonian operator $H$ of a physical system (a selfadjoint operator in Hilbert space) consists of complex nonreal generalized eigenvalues when the operator (or, more precisely, a restriction of it) is considered as acting in a rigged Hilbert space. The smallest space $\D$ of the rigged Hilbert space is determined by physical conditions (outgoing boundary conditions, labeled observables, etc.) \cite{delamadrgadella, civitagadella,jpact_book} so that the whole construction depends on the model under consideration. There are several nonequivalent definitions of the resonant spectrum (some of them are discussed also in \cite[Sect. 7.2.2, 7.2.3]{jpact_book}) but it mostly stems out from the spectrum of some extension of $H$, in the very same spirit of what we have done in this paper. The resonant (Gamow) states are eigenvectors of the extension  $H^\times$ of $H$ to $\D^\times$ (which certainly exists if $H\D\subseteq \D$). These eigenvectors certainly do not belong to $\H$ since $H$ can only have real eigenvalues; thus they are true objects in $\D^\times$. As shown before, we can perform some spectral analysis of $H$ by choosing a convenient family of interspaces $\gF_0$.  The resonant spectrum however is not contained in $\sigma^{\gF_0}(H)$ if $\gF_0$ contains $D(H)$, considered as Hilbert space with its graph norm, and $\H$.  Indeed, in this case, $\sigma^{\gF_0}(H) \subseteq \sigma_\H(H)\subseteq {\mb R}$. However, the complex eigenvalues can certainly be found in one of the sets $\sigma_{\E,\cF}(H)$ that determine the spectrum $\sigma^{\gF_0}(H)$, if the family of interspaces $\gF_0$ is sufficiently rich. This situation shows that the whole family of spectra $\{\sigma_{\E,\cF}; \E, \cF \in \gF_0\}$ should be taken into account in order to get enough information on $H$. \enrem
\section{Examples}\label{sect_examples} The following examples give some motivations to the ideas developed in the paper.
\beex \label{ex_delta} Let $\S$ and $\S^\times$ be as in Example \ref{op p}.
The operator $M_\delta$ defined by $M_\delta: \S \to \S^\times$, with
$$\ip {M_\delta f}{g} =\ip{\delta f}{g}= \ip{\delta}{f^*g}=f(0)\overline{g(0)}$$
is the {\em multiplication} operator by the Dirac $\delta$ distribution.
It is easily seen that $M_\delta \in \LSS$. We want to determine the spectrum of $M_\delta$. First we look for eigenvalues, i.e. for all $\lambda \in {\mb C}$ such that $(M_\delta - \lambda I)f=0$ has nonzero solutions.

If $\lambda$ is an eigenvalue and $f$ is an eigenvector, then
$$ \ip{M_\delta f}{g}=\lambda\ip{f}{g} , \quad \forall g \in \S;$$
i.e.,
\begin{equation}\label{star} f(0)\overline{g(0)}=\lambda\ip{f}{g}.\end{equation}
If we take $f=g$ we obtain $|f(0)|^2=\lambda \|f\|^2_2$, hence $\lambda\geq 0$.

It is clear that $\lambda=0$ is an eigenvalue. The corresponding eigenspace $\S_0:=\{f \in \S,\, f(0)=0\}$. This subspace is closed in $\S$ with its usual topology, but it is dense in $L^2({\mb R})$ with respect to its norm.

If $\lambda>0$, from \eqref{star}, we get
$$ |f(0)\overline{g(0)}|=\lambda|\ip{f}{g}|\leq \lambda \|f\|_2\|g\|_2, \quad \forall g\in \S.$$
If we choose $g_n(x)= \frac{n}{\sqrt{\pi}}\exp \{-n^2 x^2/2\}$, $n \in {\mb N}$, then $\|g_n\|=1$ and we should have
$$ \frac{n}{\sqrt{\pi}} |f(0)| \leq \lambda \|f\|, \quad n \in {\mb N};$$
then $f(0)=0$. This in turn implies that $\lambda=0$, a contradiction.
Hence, $0$ is the unique eigenvalue of $M_\delta$.

Thus, if $\lambda \in {\mb C}\setminus\{0\}$, the operator $M_\delta- \lambda I$ is injective. It is easy to check that the range $\mc{R}(M_\delta-\lambda I)$ is the following subspace of $\S^\times$
$$\mc{R}(M_\delta-\lambda I) = \{f(0)\delta -\lambda f;\, f\in \S\}.$$
There are no values of $\lambda$ for which this space contains $L^2({\mb R})$. So that  $(M_\delta -\lambda I)^{-1}$ cannot be defined on the whole $\S^\times$.

Let us consider as $\gF_0$ the family of Sobolev spaces $W^{k,2}(\mb R)$ and their duals; i.e.,
$\gF_0=\{W^{k,2}(\mb R), k\in {\mb Z}\},$
where $W^{0,2}(\mb R)=L^2(\mb R)$. From now on we shorten $W^{k,2}(\mb R)$ as $W^{k,2}$, etc.

If $k>0$, $M_\delta$ has a continuous extension to $W^{k,2}$ denoted by $M^{(k)}_\delta$ and defined by
 $$\ip {M^{(k)}_\delta f}{g} =\ip{\delta f}{g}=f(0)\overline{g(0)}, \quad \forall f,g \in W^{k,2}.$$
The continuity of every function in $W^{k,2}$ ensures that the right hand side of the previous equality is well-defined. It is easily seen that $M^{(k)}_\delta$ has no nonzero eigenvalues.
By \cite[Th. VIII.7]{brezis}, every $f\in W^{k,2}$ is bounded and then one has, for some $C>0$,
$$\left|\ip {M^{(k)}_\delta f}{g}\right|=|f(0)\overline{g(0)}|\leq \|f\|_\infty \|g\|_\infty \leq C \|f\|_{k,2} \|g\|_{r,2}.$$
This implies that $M^{(k)}_\delta f \in W^{-r,2}.$ Hence $M_\delta \in {\mc C}(W^{k,2}, W^{-r,2})$, for every $k,r>0$.
Since ${\mc R}(M^{(k)}_\delta -\lambda I)$, $\lambda \in {\mb C}$ never contains $L^2$, it follows that $\varrho_{k,-r}(M_\delta)=\emptyset$, where, for short, $ \varrho_{k,-r}(M_\delta)=\varrho_{W^{k,2},W^{-r,2}}(M_\delta)$.
On the other hand, $M_\delta \not\in {\mc C}(W^{k,2},W^{r,2}) \cup  {\mc C}(W^{-k,2},W^{-r,2})$, $k,r >0$ since, otherwise $\delta$ should be regular distribution.
In conclusion, $\sigma^{\gF_0}(M_\delta)={\mb C}$ and $0$ is the unique eigenvalue.
 \medskip

More generally, one can consider the distribution $C = \sum_{n\in {\mb Z}} \delta_n$, where $\ip{\delta_n}{f}= \overline{f(n)}$, the so-called {\em $\delta$-comb}, and the operator $M_C$ defined on $\S$ by
$$ \ip{M_C f}{g} = \ip{C}{ f^*g}= \sum_{n\in {\mb Z}} f(n)\overline{g(n)}.$$
The series on the r.h.s. converges for every $f,g \in \S$.
Every $n\in {\mb Z}$ is an eigenvalue with corresponding eigenspace $\S_n=\{f\in \S;\, f(n)=0\}$.
If we take again $\gF_0=\{W^{k,2},\, k \in {\mb Z}\}$, one finds,
also in this case, by an argument similar to the previous one,  that $\sigma^{\gF_0}(M_C)={\mb C}$.
\enex

\beex Let us consider the operator $M_\Phi$ of multiplication by a tempered distribution $\Phi$:
$$\ip{M_\Phi f}{g}= \ip{\Phi}{f^*g}, \quad f,g \in \S.$$
As shown in \cite{tratschi}, $M_\Phi \in \LSS$.

To begin with, let us first consider the case $\Phi= \Phi_h$, where $\Phi_h$ is the regular tempered distribution defined by a measurable slowly increasing function $h$; this means that there exists $m\in {\mb N}\cup\{0\}$ such that
$$ \int_{\mb R} |h(x)| (1+|x|)^{-m} dx <\infty.$$
The action of $M_{\Phi_h}$ is given by
$$\ip{M_{\Phi_h} f}{g}= \ip{\Phi_h}{f^*g} = \int_{\mb R} h(x) f(x) \overline{g (x)} dx, \quad f,g \in \S$$ and $M_{\Phi_h} \in \LSS$.

The eigenvalue equation
\begin{equation} \label{sec_eq} M_{\Phi_h} f - \lambda f=0, \quad f \in \S\end{equation}
has nonzero solutions in $\S$ if, and only if, $h$ is constant a.e. and $h(x)= a$, for almost every $x \in {\mb R}$, with $a \in {\mb C}$ and $\lambda=a$. In this case, $\sigma(M_{\Phi_h})= \sigma_p(M_{\Phi_h})= \{a\}$.

The operator $ (M_{\Phi_h} - \lambda I)^{-1}$ exists for every $\lambda \not\in \overline{h({\mb R})}$, the closure of the essential range of $h$. The operator $(M_{\Phi_h} - \lambda I)^{-1}$ can be identified with the operator of multiplication by the function $g=(h-\lambda)^{-1}$. Clearly, $(M_{\Phi_h} - \lambda I)^{-1}$ is defined on the subspace
$$ \M:= \{ \Phi \in \S^\times;\, \Phi \mbox{ is regular and } \Phi= \Phi_{(h-\lambda)f}, \, f \in \S\}. $$
Since $\M\subsetneqq \S^\times$, then $\varrho_{\S, \S^\times}(M_{\Phi_h})=\emptyset$.

Let us consider as $\gF_0$ the family of Sobolev spaces $W^{k,2}(\mb R)$ and their duals as in Example \ref{ex_delta}.

For shortness we put $\varrho_{W^{k,2}, W^{m,2}}(M_{\Phi_h}):= \varrho_{k,m}(M_{\Phi_h})$, $k,m \in {\mb Z}$.

Let us assume that $h \in L^{\infty}(\mb R)$. Then, for every $f\in W^{k,2}$, $g \in \S$
\begin{equation}\label{ineq_sobolev}|\ip{M_{\Phi_h} f}{g}|\leq \|h\|_\infty  \|f\|_{2} \|g\|_{2}\leq \|h\|_\infty \|f\|_{{k,2}}\|g\|_{{r,2}},\quad\forall k, r\in{\mb N}.\end{equation}

Hence, for every $f\in W^{k,2}$,   $M_{\Phi_h} f $ is a continuous linear functional on every $W^{r,2}$, $r\geq 0$ and, by \eqref{ineq_sobolev},
$M_{\Phi_h}\in {\mc C}(W^{k,2}, W^{-r,2})$, for every $k, r\in{\mb N}$.
The operator $(M_{\Phi_h} - \lambda I)^{-1}$ is defined on the subspace
$$ \M_{k,r}:= \{ \Phi \in W^{-r,2};\, \Phi \mbox{ is regular and } \Phi= \Phi_{(h-\lambda)f}, \, f \in W^{k,2}\}. $$
Then, we have the following situation:
\begin{description}
\item[$k=0, r=0$] In this case $\varrho_{0,0}(M_{\Phi_h})= {\mb C}\setminus\overline{h(\mb R)}$;
\item[$k>0, r=0$] The operator $M_{\Phi_h} - \lambda I$ is not onto, hence $\varrho_{k, 0}(M_{\Phi_h})=\emptyset$;
\item[$k>0, r>0$] $\M_{k,r}\subsetneqq W^{-r,2}$ and $\varrho_{k, -r}(M_{\Phi_h})=\emptyset$.
\end{description}
It remains to check the cases of indices both positive or both negative.
For this, let us assume that $h\in L^{\infty} \setminus W^{1,2}_{\rm loc}$. In this case $M_{\Phi_h}\not\in {\mc C}(W^{p,2}, W^{q,2})$ with $p, q>0$, since, if $hf \in W^{q,2}$ then $h \in W^{q,2}_{\rm loc}\subset W^{1,2}_{\rm loc}$ . By duality there cannot be a continuous extension of $M_{\Phi_h}$ belonging to ${\mc C}(W^{p,2}, W^{q,2})$ with $p, q<0$,
since, otherwise, we would have $M_{\Phi_{\overline{h}}}\in{\mc C}(W^{-q,2}, W^{-p,2})$.
Finally, we notice that no continuous extension of $M_{\Phi_h}$ belonging to ${\mc C}(W^{-p,2}, W^{q,2})$ with $p, q>0$ may exists, since, otherwise,
from $W^{p,2}\subset W^{-p,2}$ it would follow $M_{\Phi_h}g \in W^{q,2}$, for every $g \in W^{p,2}$ and this is excluded.
In conclusion, $\varrho^{\gF_0}(M_{\Phi_h})= {\mb C}\setminus\overline{h(\mb R)}$.

{Let us finally examine a case where $h$ is not bounded but slowly increasing. For instance, $h(x)=x$, $x \in {\mb R}$. We consider again as $\gF_0$ the chain of Sobolev spaces considered above. First we notice that every real number $\lambda$ is an $\gF_0$-generalized eigenvalue: indeed, the distribution $\delta_\lambda$, the Dirac delta centered at $\lambda$, is in $\in W^{-1,2}$ and one has
$$ \ip{(M_x-\lambda I) \delta_\lambda}{f}= \ip{\delta_\lambda}{(x-\lambda)f}= \lambda\overline{f(\lambda)}-\lambda\overline{f(\lambda)}=0, \quad \forall f\in W^{1,2}.$$
Since, for every $r\geq 0$
$$ |\ip{M_xf}{g}|\leq \|xf\|_2\|g\|_2\leq \|xf\|_2\|g\|_{r,2}, \quad \forall f \in \S, g \in W^{r,2},$$
$M_xf$ is, for every $f\in \S$, an element of $W^{-r,2}$, but $M_x\not\in{\mc C}(W^{k,2}, W^{-r,2})$, $k >0$. Moreover, $M_x\not\in{\mc C}(W^{k,2}, W^{r,2})$, for $k,r\geq 0$, since this would imply that the operator of multiplication by $x$, regarded as an operator in $L^2$ should be everywhere defined on $W^{k,2}$ and this is not true. By a duality argument we can also exclude that $M_x\in {\mc C}(W^{-r,2}, W^{k,2})$ and $M_x\in {\mc C}(W^{-r,2}, W^{-k,2})$, $r, k>0$.
Thus in this case $\sigma^{\gF_0}(M_x)={\mb C}$.

The situation changes if we include in $\gF_0$ the extreme spaces $\S$ and $\S^\times$: in this case the spectrum coincides with the usual spectrum $\sigma_\H(M_x)$ of $M_x$ since $M_x \in {\mc C}(\S, L^2)$ and $M_x$ is essentially selfadjoint on $\S$.}
\enex

{\beex As it is well known, the Hermite functions defined by $\phi_0(x)=\pi^{-1/4}e^{-x^2/2}$ and
$$ \phi_n(x)= (2^nn!)^{-1/2}(-1)^n \pi^{-1/4}e^{x^2/2}\left( \frac{d}{dx} \right)^n e^{-x^2}$$
constitute an orthonormal basis of $L^2({\mb R})$. If $f \in {\mc S}$, then $f$ has the expansion
\begin{equation}\label{hermite} f= \sum_{n=0}^\infty c_n\phi_n, \mbox{ with } \sup_n |c_n|n^m <\infty, \; \forall m\in{\mb N}\end{equation}and the series converges in the topology of ${\mc S}$.
 The space of sequences $\{c_n\}$ satisfying, for a given $m \in {\mb N}$,
$$\sup_n |c_n|n^m <\infty, $$
will be denoted by ${\sf s}_m$. We will indicate with ${\sf s}$ the so-called space of {\em rapidly decreasing sequences}; i.e.,
${\sf s}= \bigcap_{m \in {\mb N}}{\sf s}_m$.

An element $F\in {\mc S}^\times$ can be represented as
\begin{equation}\label{eq_seq} F= \sum_{n=0}^\infty b_n\phi_n, \mbox{ with } |b_n|\leq M(1+n)^s, \; \mbox{ for some } M>0, \, s\in{\mb N},\end{equation}
the series being weakly convergent.

Let now $\{a_n\}$ be a sequence of complex numbers such that
\begin{equation}\label{condition}
\forall \{c_n\}\in {\sf s}, \; \exists \,m \in {\mb N} \mbox{ such that } \sup_n\frac{|a_n||c_n|}{(1+n)^m}<\infty.
\end{equation}
Then
$$ f= \sum_{n=0}^\infty c_n\phi_n \mapsto Af:=\sum_{n=0}^\infty a_nc_n\phi_n $$
defines a linear map from ${\mc S}$ into ${\mc S}^\times$. Since ${\mc S}$ is a reflexive Fr\'echet space, it is sufficient to check that $A$ is continuous from ${\mc S}[\sigma({\mc S}, {\mc S}^\times)]$ into ${\mc S}^\times[\sigma({\mc S}^\times, {\mc S})]$. Continuity follows immediately from the fact that the map $$A\ad: f= \sum_{n=0}^\infty d_n\phi_n \mapsto A\ad f:=\sum_{n=0}^\infty \overline{a_n}d_n\phi_n, \quad\{d_n\} \in {\sf s},$$ is the adjoint of $A$. Hence $A\in \LSS$.
A natural choice of $\gF_0$ consists in taking the spaces $\S_m$ whose elements are all $F \in S^\times$ for which the expansion \eqref{eq_seq} has coefficients in ${\sf s}_m$ and their dual $\S_m^\times$.
It is easy to check that every $a_n$ is an eigenvalue of $A$. Thus, if $\lambda \not\in \overline{\{a_n,\, n \in {\mb N}\}}$, the sequence $\left\{\frac{1}{a_n-\lambda}\right\}$ is bounded. For these values of $\lambda$, the operator $(A-\lambda I)^{-1}$ maps $\S_m$ into $\S_m$ (and $\S_m^\times$ into $\S_m^\times$) continuously. Hence $\sigma^{\gF_0}(A)=\overline{\{a_n,\, n \in {\mb N}\}}$.\enex

\bigskip
\noindent{\bf Acknowledgement} One of us (SDB) gratefully acknowledges the warm hospitality of Prof. Maria Fragoulopoulou and of the Department of Mathematics of the University of Athens  where a part of this work was done.

\end{document}